\documentclass[a4paper,11pt,leqno]{amsart}
\usepackage{amsmath,amssymb,latexsym,cite,mathtools, upgreek}
\usepackage[colorlinks=true,urlcolor=blue,citecolor=red,linkcolor=blue,linktocpage,pdfpagelabels,bookmarksnumbered,bookmarksopen]{hyperref}
\newtheorem{theorem}{Theorem}[section]
\newtheorem{lem}[theorem]{Lemma}

\newtheorem{proposition}[theorem]{Proposition}
\newtheorem{remark}[theorem]{Remark}
\theoremstyle{definition}
\newtheorem{definition}[theorem]{Definition}

\newcommand{\g}{{G}}
\newcommand{\R}{{\bf R}}
\usepackage{cancel} 

\makeatletter
\def\cleardoublepage{\clearpage\if@twoside \ifodd\c@page\else
\hbox{}
\thispagestyle{empty}
\newpage
\if@twocolumn\hbox{}\newpage\fi\fi\fi}
\makeatother

\title{\g-convergence of elliptic and parabolic operators depending on vector fields}
\author{A. Maione}
\address{Alberto Maione: Abteilung f{\"u}r Angewandte Mathematik\\Albert-Ludwigs-Universit{\"a}t Freiburg\\Hermann-Herder-Stra{\ss}e 10\\79104 Freiburg i. Br. - Germany\\}
\email{alberto.maione@mathematik.uni-freiburg.de}
\author{F. Paronetto}
\address{Fabio Paronetto:  Dipartimento di Matematica ``Tullio Levi-Civita''\\ Universit\`a di Padova\\ via Trieste 63\\ 35121, Padova - Italy\\}
\email{fabio.paronetto@unipd.it}
\author{E. Vecchi}
\address{Eugenio Vecchi: Dipartimento di Matematica\\Università degli Studi di Bologna\\Piazza di Porta San Donato 5, 40126 Bologna, Italy}
\email{eugenio.vecchi2@unibo.it}
\thanks{A.M. is supported by the DFG SPP 2256 project ``Variational Methods for Predicting Complex Phenomena in Engineering Structures and Materials'' and by the University of Freiburg, Germany.}

\date{\today}

\begin{document}


\begin{abstract}
We consider sequences of elliptic and parabolic operators in divergence form and depending on a family of vector fields. We show compactness results with respect to $G$-convergence, or $H$-convergence, by means of the compensated compactness theory, in a setting in which the existence of affine functions is not always guaranteed, due to the nature of the family of vector fields.
\end{abstract}
\maketitle


\section{Introduction}


The asymptotic behaviour, as $h \to \infty$, from the point of view of $G$-convergence of a sequence of equations like
\begin{equation}\label{intro}
    E_h u = f
\end{equation}
in bounded domains $\Omega$ of $\mathbf{R}^n$ has been widely studied, in particular when $E_h$ is an elliptic or parabolic operator in divergence form, i.e.,
\[
E_h = -{\rm div\,}(a_h (x, \nabla))
\hskip15pt \textrm{or} \hskip15pt
E_h = {\displaystyle {{\partial_t} }} -
				{\rm div\,}(a_h (x,t,\nabla))\,,
\]
where $\nabla$ denotes the Euclidean gradient, $x \in \Omega$ and $t \in (0,T)$, with $T > 0$.

\noindent $G$-convergence was introduced for linear elliptic operators $E_h$ (for symmetric matrices $a_h$ with positive eigenvalues) by Spagnolo in a series of papers at the end of '60s
(see \cite{spagnolo1,spagnolo2,spagnolo3}) and studied later by the same author in many other papers.
Surely worthy to be recalled are \cite{deg-sp}, in the case of homogenization, where for the first time an explicit expression of the limit operator is given, \cite{col-spa} in which also a comparison between elliptic and parabolic $G$-convergence is made, and \cite{spagnolo} among the first papers about linear parabolic operators.

\noindent Regarding linear operators $E_h$ with non-symmetric matrices $a_h$, or $E_h$ nonlinear, this study involves a further difficulty: the lack of uniqueness of a representative for the limit operator, see e.g. \cite{simon}.

\noindent This problem is bypassed in the '70s by Murat and Tartar, who extended the notion of $G$-convergence, and call it $H$-convergence, to general linear and monotone elliptic operators $E_h$, see e.g. \cite{tartar-murat,T}.
We want to give a brief, and certainly not exhaustive, account regarding the case of nonlinear operators: we refer to \cite{cp-dm-df} for elliptic operators and to \cite{svan2} for parabolic ones.
Finally, we mention \cite{iquattrorussi}, where linear operator also of degree greater than two are considered, and the book \cite{pankov} where nonlinear both elliptic and parabolic operators, homogenization and random operators are presented. \\

\medskip

The aim of this paper is to extend the classical results for sequences of monotone operators to the more general setting of operators modeled on vector fields (i.e. replacing the Euclidean gradient $\nabla$ in \eqref{intro} with a family of vector fields $X$), continuing along the path traced in the recent papers \cite{M,MPSC1,MPSC2}. Before entering into the details, we want to recall that the literature concerning homogenization, $G$-convergence and integral representation of abstract functionals depending on vector fields is pretty vast, ranging from more rigid structures like Carnot groups, see e.g. \cite{BFTT, FGVN, FTT, FT, MV, M} and the references therein, up to the more general setting considered in \cite{MPSC1,MPSC2,EPV,EV}.

\noindent The setting we will take into account embraces the huge family of vector fields satisfying the H\"ormander condition.
We stress however that our results still apply to families of vector fields {\it not satisfying} the H\"ormander condition, provided that the following hypotheses are fulfilled: we consider a bounded domain $\Omega$ of ${\bf R}^n$ and a family of $m \leq n$ vector fields $X=(X_1, \ldots, X_m)$, defined and Lipschitz continuous on an open neighborhood $\Omega_0$ of $\overline{\Omega}$, such that the following conditions hold
\begin{itemize}
\item[(H1)] let $d:\R^n\times\R^n\to[0,\infty]$ be the so-called Carnot-Carath\'eodory distance function  induced by $X$, see e.g. \cite{FSSC2}. Then, $d(x,y)<\infty$ for any $x,y\in\Omega_0$, so that $d$ is a standard distance in $\Omega_0$, and $d$ is continuous with respect to the usual topology of $\R^n$;
\item[(H2)] for any compact set $K\subset\Omega_0$ there exist $r_K, C_K > 0$, depending on $K$, such that 
\[
|B_d(x,2r)|\leq C_K |B_d(x,r)|
\]
for any $x\in K$ and $r<r_K$. $B_d(x,r)$ denotes the open metric ball with respect to $d$, that is, $B_d(x,r):=\{y\in\Omega_0\,:\,d(x,y)<r\}$;
\item[(H3)] there exist geometric constants $c,C>0$ such that for every ball $B=B_{d}(\overline{x},r)$ with $cB:=B_{d}(\overline{x}, cr)\subseteq\Omega_0$, 
for every $u \in \mathrm{Lip}(c\overline{B})$ and for every $x \in \overline{B}$
\[
\left|u(x)-\frac{1}{|B|}\int_B u(y)\,dy\right|\leq C\int_{cB}|Xu(y)| \frac{d(x,y)}{|B_d(x,d(x,y))|}\,dy\,;
\]
\item[(LIC)]the $n$-dimensional vectors $X_1(x),\dots, X_m(x)$ are linearly independent for any $x\in\Omega\setminus Z_X$, where $Z_X$ is a Lebesgue measure zero subset of $\Omega$.
\end{itemize}
\medskip
As already mentioned, the main goal of the paper is to extend to the monotone, and possibly parabolic, case the result contained in \cite{MPSC2}, where the authors dealt only with the elliptic and linear case. We recall that $G$-convergence when the equations \eqref{intro} represent the Euler-Lagrange equations of a family of functionals, may be connected with $\Gamma$-convergence, see e.g. \cite{dalmaso,degiorgi-fran}. This is the approach used in the linear elliptic case in \cite{MPSC2} and it cannot be followed in the case of \textit{parabolic} problems driven by monotone operators, because the corresponding PDEs cannot be seen as Euler-Lagrange equations of appropriate functionals.
To mark once more the difference with respect to \cite{MPSC2} and to better stress the novelty of the paper, let us review the classical approach introduced by De Giorgi and Spagnolo \cite{deg-sp,spagnolo,spagnolo1,spagnolo2,spagnolo3}, which is based on the \textit{compensated compactness} and the \textit{existence of affine functions}.

\noindent\textit{\textbf{Compensated compactness}:
for any pair $(M_h)_h$, $(\nabla v_h)_h$ satisfying
\begin{align*}
    M_h\to M\text{ and }\nabla v_h\to\nabla v\quad\text{weakly in }L^2(\Omega;\mathbf{R}^n)\,,\\
    \mathrm{div}M_h=g\quad\text{for a fixed data }g\in L^2(\Omega)\,,
\end{align*}
it holds that
\begin{equation*}
    (M_h,\nabla v_h)_{\mathbf{R}^n}\to(M,\nabla v)_{\mathbf{R}^n}\quad\text{in }\mathcal{D}'(\Omega)\,.
\end{equation*}}
Roughly speaking, the compensated compactness theory ensures that the Euclidean inner product remains continuous with respect to weak convergence, even thought neither sequence is assumed to be relatively compact in $L^2(\Omega;\mathbf{R}^n)$. Here the \textit{lack of compactness} is \textit{compensated} by the boundedness of some combinations of partial derivatives.
\noindent Murat and Tartar in \cite{tartar-murat,T} extended the compensated compactness with a result involving the notion of {\it curl} (notice that $\mathrm{curl}\nabla v=0$). This theory is known as \textit{div-curl lemma} and its extension to the setting of Sobolev spaces depending on vector fields is a pretty delicate task, mainly because a proper notion of {\it intrinsic curl}, $\mathrm{curl}_X$, ensuring that $\mathrm{curl}_X Xv=0$, is not always available. We recall here that a possible notion of curl in the setting of Carnot groups has been given in \cite{FTT}, using the intrinsic complex of differential forms of Rumin.

\noindent In Theorem \ref{comp_comp} and Theorem \ref{div}, we show that in fact the classical technique due to Spagnolo is sufficient to get compensated compactness even without a proper generalization of the div-curl lemma, and in particular of the curl.

\noindent The compensated compactness is usually used to prove
the {\it closure of the class of operators in divergence form}, meaning that if
$$
A_h := -\mathrm{div} (a_h (x, \nabla)) \quad \text{$\g$-converges to } A\,,
$$
then the limit operator $A$ is in divergence form, i.e.,
\[
A = -\mathrm{div} (a (x, \nabla))\,,
\]
for some function $a$. In the Euclidean setting, the definition of $a$ goes through the existence of suitable {\it affine functions}.
\medskip

\noindent\textit{\textbf{Existence of affine functions}: for any fixed $\xi\in\mathbf{R}^n$, there exists a unique smooth enough function $u$ (at least $\mathbf{C}^2$) such that $\nabla u=\xi$.}

\noindent In the purely Euclidean framework, affine functions exist and can be represented by the Euclidean scalar product of the fixed vector $\xi$ with $x\in \mathbf{R}^n$.
Another example in which one can prove the existence of such functions is provided by the Heisenberg group. On the other hand, if we consider the case of the Grushin plane, we easily get an example of a family of vector fields satisfying our assumptions but for which affine functions may not exist. To be more precise, let $n=2$ and consider the Grushin gradient $X=(X_1,X_2)$ 
\[
X(x):=(\partial_{x},x\,\partial_{y})\,,\quad x\in\Omega_X:=\{(x,y)\in\mathbf{R}^2:\ x\neq0\}\,.
\]
Then, for any fixed $\xi=(\xi_1,\xi_2)\in\mathbf{R}^2$ with $\xi_2\neq0$, one can easily show that there exists no function $u\in\mathbf{C}^2(\Omega_X)$ such that $Xu(x)=\xi$.\\
\noindent Despite the possible non-existence of $X$-affine functions, we are however able to prove $G$-compactness by using the classical Euclidean affine functions and exploiting either the linear independence condition (LIC) on the $X$-gradient and the algebraic structure of the family $X$. See the proofs of Theorem \ref{MainTh_1} and Theorem \ref{MainTh} for the details. We finally stress that our proof drastically simplifies the one in \cite{MPSC2} for merely linear elliptic operators.

\medskip

The paper is organized as follows: in Section \ref{sect2}, we provide the functional setting of Sobolev spaces depending on vector fields and we state the main properties of the classes of monotone operators we are interested in. In Section \ref{elliptic}, we state and prove the main result in the elliptic framework and, in Section \ref{sec.MainResult}, in the parabolic setting.
Finally, in Lemma \ref{lemmafinale}, we show that the parabolic limit and the elliptic one coincide when the parabolic sequence of monotone operators is independent of time.


\section{Notations and Preliminaries}\label{sect2}


\subsection{Functional setting}



Let $X(x):=(X_1(x),\dots,X_m(x))$ be a given family of first order linear differential operators with Lipschitz coefficients on a bounded domain $\Omega$ of $\R^n$, that is,
\[
X_j(x)=\,\sum_{i=1}^nc_{ji}(x)\partial_i\quad j=1,\dots,m
\]
with $c_{ji}(x)\in\mathrm{Lip}(\Omega)$ for $j=1,\dots,m$, $i=1,\dots,n$. 

\noindent In the following, we will refer to $X$ as {\it $X$-gradient}.
As usual, we identify each $X_j$ with the vector field
\[
(c_{j1}(x),\dots,c_{jn}(x))\in\mathrm{Lip}(\Omega; \R^n)
\]
and we call 
\begin{equation*}
C(x) = [c_{ji}(x)]_{
{i=1,\dots,n}\atop{j=1,\dots,
m}}
\end{equation*}
the {\it coefficient matrix of the $X$-gradient}. 

\begin{definition}
For any $u\in L^1(\Omega)$ we define $Xu$ as an element of ${\mathcal D}'(\Omega;\R^m)$ as follows
\begin{equation}\label{defXjT}
\begin{split}
\langle Xu, \psi \rangle_{\mathcal{D}'\times\mathcal{D}}:&=(\langle X_1u, \psi_1 \rangle_{\mathcal{D}'\times\mathcal{D}},\dots,\langle X_mu, \psi_m \rangle_{\mathcal{D}'\times\mathcal{D}})\\
         &=-\int_\Omega u\left(\sum_{i=1}^n\partial_i( c_{1i}\,\psi_1),\dots, \sum_{i=1}^n\partial_i( c_{mi}\,\psi_m)\right)dx
\end{split}
\end{equation}
for any $\psi=(\psi_1,\dots,\psi_m)\in\mathcal{D}(\Omega;\R^m)=\mathbf{C}_c^\infty(\Omega;\R^m)$.
\end{definition}
If we set $X^T\psi:=(X^T_1\psi_1,\dots, X^T_m\psi_m)$, with
\begin{equation*}
X_j^T\varphi:=\,-\sum_{i=1}^n\partial_i( c_{ji}\,\varphi)=\,-\left(\mathrm {div}(X_j)+X_j\right)\varphi
\end{equation*}
for any $\varphi\in\mathbf{C}^\infty_c(\Omega)$ and $j=1,\dots,m$, then \eqref{defXjT} becomes
\[
\langle Xu, \psi \rangle_{\mathcal{D}'\times\mathcal{D}}=\,\int_\Omega u\, X^T\psi\,dx\quad\text{for any }\psi \in\mathbf{C}_c^\infty(\Omega;\R^m)\,.
\]
\begin{remark}
Notice that if $X=\nabla=\,(\partial_1,\dots,\partial_n)$, then
\begin{equation*}
X_j^T\varphi=\,-\partial_j \varphi\quad\text{for any }j=1,\dots,n\,.
\end{equation*}
\end{remark}
Let $a: \Omega\times \R^{m}\to \R^{m}$ as assume that $a(\cdot,X)$ is smooth enough in $\Omega$.
The operator $X$-divergence of $a(\cdot,X)$ is defined by
\begin{equation*}
\,{\rm div}_X(a(x,X)):=\sum_{j,i=1}^mX_j^T(a(x,X_i))\,,\quad x\in\Omega
\end{equation*}
and its domain is the set $W^{1,p}_{X,0}(\Omega)$ defined as follows.
\begin{definition}\label{Definition 1.1.1}
We define the anisotropic Sobolev spaces in the sense of Folland and Stein \cite{FS} as
\[
W_X^{1, p}(\Omega):=\left\{u\in L^p(\Omega):X_j u\in L^p(\Omega)\ {\hbox{\rm for }} j=1,\dots,m\right\}.
\]
These spaces, endowed with the norm
\[
\|u\|_{W^{1,p}_X(\Omega)}:=\left(\int_{\Omega} |u|^p\, dx + \int_{\Omega} |Xu|^p\, dx\right)^\frac{1}{p},
\]
are Banach spaces for any $1\leq p \leq \infty$, reflexive if $1< p < +\infty$.

\noindent Moreover, we denote $W^{1,p}_{X,0}(\Omega)$ the closure of $\mathbf{C}^1_c(\Omega)\cap {W^{1,p}_X(\Omega)}$ in $W^{1,p}_X(\Omega)$.
\end{definition}
Since vector fields $X_j$ have  Lipschitz continuous coefficients, then, by definition,
\begin{equation}\label{inclclassSobsp}
W^{1,p}(\Omega)\subset W^{1,p}_X(\Omega)\quad\forall\,p\in [1,\infty]
\end{equation}
and, for any $u\in W^{1,p}(\Omega)$,
\begin{equation*}
Xu(x)=\,C(x)\,Du(x)\quad\text{ for a.e. }x\in\Omega\,.
\end{equation*}
Here $W^{1,p}(\Omega)$ denotes the classical Sobolev space, or, equivalently, the space $W^{1,p}_X(\Omega)$ associated to the family $X=\,(\partial_1,\dots,\partial_n)$.
Inclusion \eqref{inclclassSobsp} can be strict, in particular when the number of vector fields is strictly less than the dimension of the space, and turns out to be continuous.
By the Lipschitz regularity assumption, the validity of the classical result \lq$H=W$\rq\, of Meyers and Serrin \cite{MS} is still guaranteed as proved, independently, in \cite{FSSC1} and \cite{GN}.
\medskip

As a consequence of conditions (H1) - (H3), it has been proved in \cite{FSSC2,MPSC2} the validity of a Rellich-type theorem and a Poincar\'e inequality in $W^{1,p}_{X,0}(\Omega)$.
\begin{theorem}[{\cite[Theorem 3.4]{FSSC2}}]\label{immersion}
Let $1\leq p<\infty$ and let $X$ satisfy conditions \text{\rm (H1), (H2)} and  \text{\rm (H3)}. Then, $W^{1,p}_{X,0}(\Omega)$ 
compactly embeds in $L^p(\Omega)$.
\end{theorem} 
\begin{proposition}[{\cite[Proposition 2.16]{MPSC2}}]\label{PoincareW1p0}
Under the hypotheses of the previous theorem and also assuming that $\Omega$ is connected, there exists a positive constant $c_{p,\Omega}$, depending only on $p$ and $\Omega$, such that 
\begin{equation*}
c_{p,\Omega}\,\int_\Omega |u|^p\,dx\leq\, 
\int_\Omega |Xu|^p\,dx\quad\text{for any }u\in W^{1,p}_{X,0}(\Omega)
\end{equation*}
and
\begin{equation*}
\|u\|_{W^{1,p}_{X,0}(\Omega)}:=\left(\int_{\Omega} |Xu|^p\, dx\right)^{\frac{1}{p}}
\end{equation*}
is a norm in $W^{1,p}_{X,0}(\Omega)$ equivalent to $\|\cdot\|_{W^{1,p}_X(\Omega)}$.
\end{proposition}


\subsection{Space-time setting}


Let $(p,p')$ be a H\"older's conjugate pair, with  $p \geq 2$.
We denote
\begin{align*}
V := W^{1,p}_{X,0}(\Omega) \,, \qquad H := L^2(\Omega) \,, \qquad V' := W^{-1,p'}_X (\Omega) \,,
\end{align*}
where $V'$ is the dual space of $V$, and we identify the dual space of $H$, $H'$, with $H$ itself, in such a way that
\begin{align}
\label{immersioni}
V \subset H \subset V'
\end{align}
where the embeddings are dense and continuous. In a similar way, we define
\begin{align*}
\mathcal{V} := L^p(0,T;V)\,,  \qquad \mathcal{H} := L^2(0,T;H) \,, \qquad \mathcal{V}' := L^{p'}(0,T;V') 
\end{align*}
and, by \eqref{immersioni}, we have
\[
\mathcal{V} \subset \mathcal{H} \subset \mathcal{V}'
\]
with continuous and dense embeddings. We endow $\mathcal{V}$ and $\mathcal{H}$, respectively, with the following norms:
\begin{align*}
\| u \|_\mathcal{V} := \left( \int_0^T \| u(t) \|_V^p \, dt \right)^\frac{1}{p} ,		\qquad
\| u \|_\mathcal{H} := \left( \int_0^T \| u(t) \|_H^2 \,dt \right)^\frac{1}{2} ,
\end{align*}
and $\mathcal{V}'$ with the natural norm
\begin{align*}
\| u \|_{\mathcal{V}'}:=\left(\int_0^T\|u(t)\|_{V'}^{p'}\,dt\right)^\frac{1}{p'} .
\end{align*}

\begin{definition}
\label{derivata}
Given two Banach spaces $Y_1, Y_2$, we say that $v \in L^1(0,T; Y_2)$ is the generalized derivative of $u \in L^1(0,T; Y_1)$ if
\begin{align*}
 \int_0^T v(t) \varphi (t) \, dt = - \int_0^T u(t) \varphi'(t) \, dt
\end{align*}
for every $\varphi \in \mathbf{C}^\infty_c ((0,T); \R)$.
\end{definition}

\noindent
Notice that the integral in the previous definition is the Bochner integral and that $ \int_0^T v(t) \varphi (t) \, dt$, $\int_0^T u(t) \varphi'(t) \, dt \in Y_1 \cap Y_2$ (see, for instance, \cite[Chapter 23]{Z2A} or \cite[Chapter 3]{show2}). \\ [0.4em]
The natural space of solutions to parabolic PDEs is the Banach space
\begin{align*}
\mathcal{W} := \{u\in\mathcal{V} \, | \, u' \in \mathcal{V}'\}\,,
\end{align*}
endowed with the norm
\[
\| u \|_{\mathcal{W}} := \| u \|_{\mathcal{V}} + \| u'\|_{\mathcal{V}'}\,.
\]
$u'$ is to be intended as the generalized derivative of $u$ (see Definition \ref{derivata}).

\noindent In the following proposition we recall some results regarding the space $\mathcal{W}$, see, e.g., Proposition 1.2, Corollary 1.1 and Proposition 1.3 in \cite{show2}.
\begin{proposition}
\label{propW}
The space $\mathcal{W}$ continuously embeds in $\mathbf{C}^0 ([0,T] ; H)$ and, for every $u, v \in \mathcal{W}$ and for every $t, s \in [0,T]$, the following generalized integration by parts formula holds:
\begin{align*}
(u(t),v(t))_{H} & - (u(s),v(s))_{H} =\\
& = \int_s^t \langle  u'(\tau),v(\tau)\rangle_{V'\times V} \,d \tau + \int_s^t \langle  v'(\tau),u(\tau)\rangle_{V'\times V}\, d\tau\,.
\end{align*}
Moreover, the space $\mathcal{W}$ compactly embeds in $L^p (0,T ; H)$.
\end{proposition}
Now consider the operators
\[
A : V \to V' , \hskip30pt \mathcal{A} : \mathcal{V} \to \mathcal{V}'
\]
and
\[
\mathcal{P}: \mathcal{W} \to \mathcal{V}' \,, \quad \mathcal{P} u := u' + \mathcal{A} u\,.
\]
\begin{definition}\label{soluzione}
Let $g \in V'$. 
We define $u \in V$ a solution to the problem
\begin{equation}\label{Elll}
A u = g \qquad \text{in } V'
\end{equation}
if
\begin{align*}
\langle A u , v \rangle_{V' \times V} = \langle g, v \rangle_{V'\times V}\quad\text{for every }v \in V.
\end{align*}
Moreover, if $f\in\mathcal{V}'$ and $\varphi\in H$, we define $u \in \mathcal{W}$ a solution to the problem
\begin{equation}\label{prob}
\begin{cases}
u' + \mathcal{A} u = f \quad&\text{in }\mathcal{V}'	\\
u(0) = \varphi\quad&\text{in }H
\end{cases}
\end{equation}
if
\begin{align*}
    \langle u'(t),v\rangle_{V'\times V}+\langle \mathcal{A}u(t),v\rangle_{V'\times V}=\langle f(t),v\rangle_{V'\times V}
\end{align*}
for a.e. $t \in (0,T)$, for every $v \in V$ and if $u(0) = \varphi$ in $H$.
\end{definition}

\noindent
We conclude this part by recalling a result useful for the sequel.
\begin{lem}[{\cite[Lemma 7.8]{cp-dm-df}}]
\label{chiado'}
Let $U$ be a bounded open set in $\R^k$, $\vartheta_1, \dots \vartheta_m$ be non-negative numbers such that
$\vartheta_1 + \dots + \vartheta_m \leq 1$, and assume that $(r_{1,h})_h, \dots, (r_{m,h})_h$
and $(s_h)_h$ are sequences in $L^1(U)$ such that, for any $i$
\[
r_{i,h} \geq 0 \quad\text{and}\quad
|s_h| \leq {r_{1,h}^{\vartheta_1}} \cdot \, \dots \, \cdot {r_{m,h}^{\vartheta_m}}
\quad \text{a.e. in } U \text{ for  every }h\in{\bf N}\,.
\]
Moreover, assume the existence of $r_1, \dots, r_m, s \in L^1(U)$ such that
\[
r_{i,h} \to r_i\quad\text{and} \quad s_h \to s\quad \textit{in } \mathcal{D}' (U)
\]
as $h\to\infty$, for any $i=1, \dots m$. Then,
\[
|s| \leq r_1^{\vartheta_1} \cdot \, \dots \, \cdot r_m^{\vartheta_m} \quad \textit{a.e. in } U\,.
\]
\end{lem} 


\subsection{Position of the problems}\label{Problema}


\begin{definition}\label{M,alpha,beta}
Let $\alpha\leq\beta$ be positive constants. 
We define $\mathcal{M}_{\Omega\times(0,T)}(\alpha,\beta,p)$ the class of Carath\'eodory functions $a: \Omega\times(0,T)\times \R^{m}\to \R^{m}$ satisfying
\begin{itemize}
	\item [$(i)$] $a (x,t,0)=0$;
	\item [$(ii)$] $\left( a(x,t,\xi) - a(x,t,\eta), \xi-\eta\right)_{\R^m}\geq \alpha |\xi-\eta|^p$;
	\item [$(iii)$] $|a (x,t,\xi) - a(x,t,\eta)| \leq \beta\left[1+|\xi|^p+|\eta|^p\right]^\frac{p-2}{p}|\xi-\eta|$
\end{itemize}
for a.e. $(x,t)\in\Omega\times(0,T)$ for every $\xi,\eta\in \R^{m}$.
\medskip

\noindent
We also denote by $\mathcal{M}_{\Omega}(\alpha,\beta,p)$ the subclass of $\mathcal{M}_{\Omega \times (0,T)}(\alpha,\beta,p)$ of functions $a$ independent of $t$.
\end{definition}

\begin{remark}
\label{chebelfreschetto}
If $a \in \mathcal{M}_{\Omega\times(0,T)}(\alpha,\beta,p)$, then
\begin{itemize}
    \item [$(iii)'$] $|a(x,t,\xi) - a(x,t,\eta)| \leq \beta' \left[1+|\xi|^p+|\eta|^p\right]^{\frac{p-2}{p-1}} |\xi-\eta|^{\frac{1}{p-1}}$
\end{itemize}
for a.e. $(x,t) \in \Omega \times (0,T)$ for every $\xi, \eta \in \R^{m}$ and for some $\beta' \geq \beta$. 
\medskip

\noindent Indeed, by $(i)$, $(ii)$, $(iii)$ and Cauchy-Schwarz inequality, we get
\begin{align*}
    |a(x,t,\xi) - a(x,t,\eta)| &\leq \beta \left[1+|\xi|^p+|\eta|^p\right]^{\frac{p-2}{p}} \alpha^{-\frac{1}{p}} (a(x,t,\xi)-a(x,t,\eta), \xi - \eta)_{\R^m}^{\frac{1}{p}}\\
    &\leq\alpha^{-\frac{1}{p}}\beta\left[1 + |\xi|^p + |\eta|^p\right]^{\frac{p-2}{p}}|a(x,t,\xi)-a(x,t,\eta)|^{\frac{1}{p}}|\xi - \eta|^{\frac{1}{p}}\,,
\end{align*}
i.e.,
\begin{align*}
    |a(x,t,\xi) - a(x,t,\eta)|^{\frac{p-1}{p}} \leq\alpha^{-\frac{1}{p}}\beta \left[1+|\xi|^p+|\eta|^p\right]^{\frac{p-2}{p}}|\xi - \eta|^{\frac{1}{p}}\,.
\end{align*}
The thesis follows choosing $\beta'\geq \big( \alpha^{-1}\,\beta^p\big)^\frac{1}{p-1}$.
\end{remark}

\begin{definition}\label{M,alpha,beta'}
We denote $\tilde{\mathcal{M}}_{\Omega\times(0,T)}(\alpha,\beta',p)$ and $\tilde{\mathcal{M}}_{\Omega}(\alpha,\beta',p)$, respectively, the class of Carath\'eodory functions $a: \Omega\times(0,T)\times \R^{m}\to \R^{m}$ satisfying $(i)$, $(ii)$ and $(iii)'$ and its subclass of functions independent of $t$. By Remark \ref{chebelfreschetto},
\[
\mathcal{M}_{\Omega\times(0,T)}(\alpha,\beta,p)\subset\tilde{\mathcal{M}}_{\Omega\times(0,T)}(\alpha,\beta',p).
\]
\end{definition}

\begin{remark}
\label{wembledon}
Let $a \in \mathcal{M}_{\Omega\times(0,T)}(\alpha,\beta,p)$ and define the operators
\[
A(t) : V \to V' ,	\quad A(t) u := \mathrm{div}_X( a (x,t,Xu))\,,\quad t \in [0,T]
\]
and
\[
\mathcal{A} : \mathcal{V} \to \mathcal{V}' ,	\quad \mathcal{A} u := \mathrm{div}_X( a(x,t,Xu))	\,.
\]
Notice that $\mathcal{A} u (t) = A(t) u(t)$ for $u \in \mathcal{V}$ and that $[0,T] \ni t \mapsto \langle A(t)u,v\rangle_{V'\times V}$ is measurable for every $u, v \in V$. 
\end{remark}
\begin{theorem}[{\cite[Theorem 26.A]{Z2A}}]
\label{existencee}
Let $a \in \mathcal{M}_{\Omega}(\alpha,\beta,p)$ (or, equivalently, $a \in \tilde{\mathcal{M}}_{\Omega}(\alpha,\beta',p)$) and define $A u := \mathrm{div}_X (a (x, Xu (x)))$.
Then, for every $g \in V'$ there exists a unique solution $u\in V$ to problem \eqref{Elll}. Moreover
\begin{align*}
\| u \|_{V} \leq \alpha^{-\frac{1}{p-1}}  \| g \|_{V'}^{\frac{1}{p-1}} \, .
\end{align*}
\end{theorem}
\begin{theorem}[{\cite[Theorem 30.A]{Z2A}}]
\label{existence}
Let $a \in \mathcal{M}_{\Omega \times (0,T)}(\alpha,\beta,p)$ (or, equivalently, $a \in \tilde{\mathcal{M}}_{\Omega \times (0,T)}(\alpha,\beta',p)$) and define $\mathcal{A} u := \mathrm{div}_X (a (x, t, Xu (x)))$.
Then, for every $f \in \mathcal{V}'$ and for every $\varphi \in H$ there exists a unique solution $u\in\mathcal{W}$ to problem \eqref{prob}.
Moreover, for $a \in \mathcal{M}_{\Omega \times (0,T)}(\alpha,\beta,p)$, there exists a positive constant $c$, depending only on $\alpha, \beta$ and $p$, such that
\begin{align*}
\| u \|_{\mathcal{W}} \leqslant c 
	\Big[ \| f \|_{\mathcal{V}'} + \big( 1 + \| f \|_{\mathcal{V}'}^{\frac{1}{p-1}} + \| \varphi \|_{H}^{\frac{2}{p}} \big)^{p-2} \big( \| f \|_{\mathcal{V}'}^{\frac{1}{p-1}} + \| \varphi \|_{H}^{\frac{2}{p}} \big) \Big] \,.
\end{align*}
\end{theorem}

Given $(a_h)_h \subset \mathcal{M}_{\Omega}(\alpha,\beta,p)$ and $a \in \mathcal{M}_{\Omega}(\alpha',\beta',p)$, for some positive constants $\alpha\leq\beta$, $\alpha'\leq\beta'$, denote
\begin{equation}\label{operatori_ellittici}
\begin{split}
    A_h : V \to V'\,,&\quad A_h u := \mathrm{div}_X (a_h (x, Xu (x)))\,,\\
    A : V \to V'\,,&\quad A u := \mathrm{div}_X (a (x, Xu (x)))\,.
\end{split}
\end{equation}
Fix $g \in V'$ and let $u_h, u \in V$ be, respectively, the unique solutions to
\begin{align}
\label{E_h}\tag{$E_h$}
A_h u &= g	\quad \text{in }{V}'\,,\\
\label{E}\tag{$E$}
A u &= g \quad \text{in }{V}' .
\end{align}
\begin{definition}
We say that $A_h$ $\g$-converges to $A$ if for every $g \in V'$
\begin{align*}
    u_h \to u\quad &\text{strongly in }	L^p (\Omega)\,,\\
    a_h (\cdot, Xu_h) \to a (\cdot, Xu) \quad  &\text{weakly in }	L^{p'} (\Omega ; \R^m) \,.
\end{align*}
\end{definition}
In a similar way, let $(a_h)_h \subset \mathcal{M}_{\Omega\times(0,T)}(\alpha,\beta,p)$ and $a \in \mathcal{M}_{\Omega\times(0,T)}(\alpha',\beta',p)$, for some positive constants $\alpha\leq\beta$, $\alpha'\leq\beta'$, and denote
\begin{equation}\label{operatori}
\begin{split}
    A_h (t) : V \to V'\,,\quad&	A_h (t) u := \mathrm{div}_X (a_h (x, t, Xu (x)))\,,\\
    A (t) : V \to V'\,,\quad&	A (t) u := \mathrm{div}_X (a (x, t, Xu (x)))\,,\\
    \mathcal{A}_h : \mathcal{V} \to \mathcal{V}'\,,\quad&	\mathcal{A}_h u := \mathrm{div}_X (a_h (x, t, Xu (x,t)))\,,\\
    \mathcal{A} : \mathcal{V} \to \mathcal{V}'\,,\quad&	\mathcal{A} u := \mathrm{div}_X (a (x, t, Xu (x,t)))\,,\\
    \mathcal{P}_h : \mathcal{W} \to \mathcal{V}'\,,\quad&	\mathcal{P}_h u := u' + \mathcal{A}_h u\,,\\
    \mathcal{P} : \mathcal{W} \to \mathcal{V}'\,,\quad&	\mathcal{P} u := u' + \mathcal{A} u\,.
\end{split}
\end{equation}

\noindent
Fix $f \in \mathcal{V}'$ and $\varphi\in H$, and let $u_h,u \in \mathcal{W}$ be, respectively, the unique solutions to
\begin{align}
\label{P_h}
\begin{cases}\tag{$P_h$}
\mathcal{P}_h u = f	&	\text{ in }	\mathcal{V}'	\\
u(0) = \varphi		&	\text{ in }	H
\end{cases}\,,\\
\label{P_infty}
\begin{cases}\tag{$P$}
\mathcal{P} u = f	&	\text{ in } \mathcal{V}'	\\
u(0) = \varphi		&	\text{ in } H
\end{cases}
\, .
\end{align}
\begin{definition}
We say that $\mathcal{P}_h$ $\g$-converges to $\mathcal{P}$ if for every $f \in \mathcal{V}'$ and $\varphi \in H$
\begin{align*}
    u_h\to u\quad&\text{strongly in }L^p(0,T;H)\,,\\
    a_h (\cdot,\cdot, Xu_h)\to a(\cdot,\cdot, Xu)\quad&\text{weakly in }	L^{p'}(0,T;L^{p'}(\Omega;\R^m))\,.
\end{align*}
\end{definition}
We conclude this section with a result that will be useful in the proof of Lemma \ref{lemmafinale}. The proof is left to the reader.
\begin{remark}
\label{ultimanota}
Consider the sequence of problems \eqref{E_h} and denote by $u_h (g)$ their solutions.
Suppose that $A_h$ $G$-converges to $A$. Then, it holds that
\begin{align*}
    u_h (g_h)\to u(g)\quad&\text{strongly in }L^p (\Omega)\,,\\
    a_h (\cdot, Xu_h (g_h))\to a(\cdot, Xu(g))\quad&\text{weakly in }L^{p'}(\Omega ; \R^m)\,, 
\end{align*}
provided that $(g_h)_h \subset V'$ strongly converges to $g$ in $V'$. \\
The analogous holds for problems \eqref{P_h} where one considers two sequences of data $(f_h)_h$ and $(\varphi_h)_h$, provided that
$f_h \to f$ strongly in $\mathcal{V}'$ and $\varphi_h \to \varphi$ strongly in $H$.
\end{remark}
\section{Elliptic \g-convergence}\label{elliptic}
In this section we state and prove a $G$-compactness result in the elliptic case, namely Theorem \ref{MainTh_1}.
In all this section we always assume that $\Omega$ is a bounded domain of $\mathbf{R}^n$, that $2\leq p<\infty$ and that $X$ satisfies conditions (H1), (H2), (H3) and (LIC), given in the Introduction.
\begin{theorem}
\label{comp_comp}
Let $v_h,v \in V$ and $M_h,M \in L^{p'}(\Omega; \R^m)$ satisfy, respectively,
\begin{align}\notag
v_h \to v \quad&\text{weakly in }V\,,\\
\label{convmom}
M_h \to M \quad&\text{weakly in }L^{p'}(\Omega;\R^m)
\end{align}
and assume that
\begin{equation}\label{mo'mifacciounadoccia}
\begin{split}
\text{\rm div}_X \, M_h = g \quad \text{in } V'\text{ for some }g \in V'.
\end{split}
\end{equation}
Then,
\[
(M_h ,X v_h)_{\R^m} \to (M, Xv)_{\R^m}\quad\text{in }\mathcal{D}'(\Omega) \, .
\]
\end{theorem}
\begin{proof}
Fix $\varphi \in C^{\infty}_c (\Omega)$ and consider the quantities $\langle\text{\rm div}_X M_h,\varphi\rangle_{V'\times V}$ and $\langle\text{\rm div}_X M_h,v_h\varphi\rangle_{V'\times V}$.
By \eqref{convmom} and \eqref{mo'mifacciounadoccia}, we get
\begin{align}
\label{divergenza}
\text{\rm div}_X M = g  \hskip10pt \textrm{in } {\mathcal D}'(\Omega)\,.
\end{align}
Moreover,
\begin{align*}
    \int_{\Omega} (M_h, X v_h)_{\R^m}\varphi \, dx=\langle \text{\rm div}_X M_h, v_h \varphi \big\rangle_{V'\times V} - \int_{\Omega}  (M_h, X \varphi)_{\R^m} v_h \, dx\,.
\end{align*}
Consider the right hand side terms. By assumptions we have that
\[
\lim_{h \to +\infty} \big\langle \text{\rm div}_X M_h, v_h \varphi \big\rangle_{V'\times V} = \big\langle g, v \varphi \big\rangle_{V'\times V} 
\]
and, since $(v_h)_h$ strongly converges to $v$ in $L^p (\Omega)$ (see Theorem \ref{immersion}), we get
\[
\lim_{h \to +\infty} {\displaystyle \int_{\Omega}}  (M_h, X \varphi)_{\R^m} v_h \, dx
\int_{\Omega} (M, X \varphi)_{\R^m} v \, dx \,,
\]
that is,
\begin{align*}
\lim_{h \to +\infty}\int_{\Omega} (M_h, X v_h)_{\R^m} \varphi \, dx = \big\langle g , v \varphi \big\rangle_{V'\times V} - \int_{\Omega} (M, X \varphi)_{\R^m} v \, dx  \, .
\end{align*}
Then, the thesis follows by \eqref{divergenza}.
\end{proof}

\begin{lem}
\label{lem6.5_ell}
Fix $g \in V'$ and denote by $u_{h} (g)$ the solution to problem \eqref{E_h}.
Then, there exist two continuous operators
\begin{align*}
B: V' &\to V\,,\\
M: V' &\to L^{p'} (\Omega ; \R^m)
\end{align*}
such that, up to subsequences, the following convergences hold
\begin{align*}
    u_{h} (g)\to B(g)\quad&\text{strongly in }L^p(\Omega)\,,\\
    a_h (\cdot, X u_{h}(g))\to M (g)\quad&\text{weakly in }L^{p'}(\Omega;\R^m)\,.
\end{align*}
Moreover, $B$ is invertible and $M$ satisfies
\begin{align}\label{annina}
\text{\rm div}_X M (g) = g	\quad\textit{in } V'\quad\text{for every }g \in V'\,,
\end{align}
\begin{equation}\label{conditionM}
\begin{split}
    | M(f) - M(g) | &\leq \beta' (1 + (M(f), X B(f))_{\mathbf{R}^m} + (M(g), X B(g))_{\mathbf{R}^m}) ^{\frac{p-2}{p - 1}} \\
    &\quad \times| X B(f) - X B (g) |^{\frac{1}{p-1}}
\end{split}
\end{equation}
for every $f,g \in V'$, with $\beta' = \big( \beta \, \alpha^{-\frac{1}{p}} \max\{ 1 , \alpha^{-\frac{p-2}{p}} \} \big)^{\frac{p}{p-1}}$.
\end{lem}
\begin{proof}
By Theorem \ref{existencee}, the sequence $(u_h (g))_h$ is bounded in $V$ and then, up to a subsequence, $(u_h (g))_h$ weakly converges in $V$. By Theorem \ref{immersion}, and up to a further subsequence, it strongly converges in $L^p(\Omega)$.

Fix $g \in V'$ and define
\begin{align*}
B(g) := \lim_{h\to\infty} u_h (g)\,.
\end{align*}
Since, by Definition \ref{M,alpha,beta} (iii), 
\[
| a_h (x, X u_h(g)) | \leq \beta\left[1+| X u_h(g) |^p \right]^\frac{p-2}{p} | X u_h(g) | \leq \beta\left[1+| X u_h(g) |^p \right]^\frac{p-1}{p}
\]
a.e. $x\in\Omega$, then
\begin{align*}
    \int_{\Omega} | a_h (x, X u_h(g)) |^{\frac{p}{p-1}} \, dx\leq \beta^{\frac{p}{p-1}} | \Omega| + \beta^{\frac{p}{p-1}} \int_{\Omega} | X u_h(g) |^p \, dx \,,
\end{align*}
that is, the sequence $(a_h (\cdot, X u_h (g)))_h$ turns out to be bounded in $L^{p'} (\Omega ; \R^m)$.
Therefore, for every $g \in V'$ and up to a subsequence, there exists a limit $M(g) \in L^{p'} (\Omega ; \R^m)$ such that $(a_h (\cdot, X u_h (g)))_h$ weakly converges to $M(g)$ in $L^{p'} (\Omega ; \R^m)$.
Moreover, since
\begin{align*}
\langle g, \varphi \rangle_{V' \times V} = \langle A_h u_h (g), \varphi \rangle_{V' \times V} = \int_{\Omega} \big( a_h (x, X u_h(g)), X \varphi \big)_{\mathbf{R}^m} \, dx
\end{align*}
for any $\varphi \in C^{\infty}_c (\Omega)$ by hypotheses, then
\begin{align*}
\int_{\Omega} \big( M(g), X \varphi \big)_{\mathbf{R}^m} \, dx = \langle g, \varphi \rangle_{V' \times V}\,,
\end{align*}
that is, \eqref{annina} holds.

Fix $f, g \in V'$. By Definition \ref{M,alpha,beta}, it holds that
\begin{align*}
    |a_h(x,Xu_h(f)) &- a_h(x, Xu_h(g))|\leq		\beta (1+ | Xu_h(f)|^p + |Xu_h(g)|^p )^{\frac{p-2}{p}}\\
    &\quad\times|X u_h(f) - Xu_h(g)|\\
    &\leq \beta\alpha^{-1/p} ( 1+ \alpha^{-1} (a_h(x,X u_h (f)), Xu_h(f))_{\mathbf{R}^m} \\
    &\quad+ \alpha^{-1} (a_h (x, Xu_h(g)), Xu_h(g)))_{\mathbf{R}^m} )^{\frac{p-2}{p}} \\
    &\quad\times  (a_h (x, Xu_h(f)) - a_h(x,Xu_h(g)), Xu_h(f) -  Xu_h(g))_{\mathbf{R}^m}^{1/p}
\end{align*}
a.e. $x\in\Omega$ and, letting $h\to\infty$, we get
\begin{align*}
    | M(f) - M(g) | &\leq \beta\alpha^{-1/p} \max\{ 1 , \alpha^{-\frac{p-2}{p}} \} \big( 1 + (M(f), X B(f))_{\mathbf{R}^m}\\
    &\quad+ (M(g), X B(g))_{\mathbf{R}^m} \big)^{\frac{p-2}{p}}	\\
	&\quad\times\big( M(f) - M(g), X B(f) - X B (g) \big)_{\mathbf{R}^m}^{1/p}\,,
\end{align*}
by means of Lemma \ref{chiado'} and Theorem \ref{comp_comp}. Thus, \eqref{conditionM} follows by Cauchy-Schwarz inequality.

We conclude by showing the invertibility of $B$.
Fix $f, g \in V'$ and assume that $B (f) = B (g)$.
By \eqref{conditionM}, it holds that $M (f) = M(g)$ and, by \eqref{annina}, we conclude that $f = g$.
Moreover, $B (V')$ is dense in $V$. In fact, if $g_o \in V'$ satisfies
\[
\langle g_o , B(g) \rangle_{V' \times V} = 0 \quad \text{for every } g \in V'
\]
then, in particular, $\langle g_o , B(g_o) \rangle_{V' \times V} = 0$ and, by Definition \ref{M,alpha,beta} (ii) and the lower semicontinuity of the norm $\|\cdot\|_V$, we get
\begin{align*}
    0 & = \langle g_o , B(g_o) \rangle_{V' \times V} = \langle g_o , \lim_{h\to\infty}u_h(g_o) \rangle_{V' \times V}\\
	& = \lim_{h \to +\infty} \langle A_h u_h (g_o) , u_h (g_o) \rangle_{V' \times V}\\
	& \geq \liminf_{h \to +\infty} \alpha \| u_h (g_o) \|_V^p =  \alpha \| B (g_o) \|_V^p \geq0\,.
\end{align*}
The conclusion follows by the injectivity of $B$ and since $B(0) = 0$.

The operator $B^{-1} : B(V') \to V'$ can be uniquely extended to an operator $B^{-1}:V'\to V$, by the density of $B(V')$ in $V$.
\end{proof}

\begin{theorem}\label{MainTh_1}
Consider a sequence $(a_h)_h\subset\mathcal{M}_{\Omega}(\alpha,\beta,p)$ and the related sequence of elliptic operators $(A_h)_{h}$, defined in \eqref{operatori_ellittici}.

\noindent Then, there exists $a \in \tilde{\mathcal{M}}_{\Omega}(\alpha, \beta', p)$,
with $\beta' = \left( \alpha^{-1} \beta^p \right)^{\frac{1}{p-1}}  \max \{ 1 , \alpha^{-\frac{p-2}{p-1}} \}$, such that, up to subsequences,
\[
A_h\quad \g\text{-converges to } A\,,
\]
where $A:V\to V'$ is the operator in $X$-divergence form associated to $a$ and defined by
\[
A(u)=\mathrm{div}_X (a (\cdot, Xu ))\quad\text{for any }u\in V\,.
\]
\end{theorem}
\begin{proof}
Let us divide the proof of Theorem \ref{MainTh_1} in three steps.
\medskip

\noindent\textbf{Step 1.} Construction of the limit operator and useful estimates.

Let $B$ and $M$ be the operators introduced in Lemma \ref{lem6.5_ell} and define
\[
A : V \to V'\,,\quad A := B^{-1}.
\]
By \eqref{annina},
\[
\text{\rm div}_X M (g) = g = A (B (g)) \quad \text{for every } g \in V'.
\]
Moreover, define $N : V \to L^{p'}(\Omega; \R^m)$ by $N := M \circ A$.

\noindent As a consequence of Theorem \ref{comp_comp} and Lemma \ref{lem6.5_ell}, it easy to show that $N$ satisfies
\begin{equation}\label{pastaconlabottarga}
    \alpha | X v - X w |^p \leq \big\langle N (v) - N (w) ,  Xv - Xw \big\rangle_{V' \times V} \, ,
\end{equation}
\begin{equation}\label{pasta2}
    | N (v) - N (w) | \leq \beta' \big( 1 + (N(v), X v)_{\mathbf{R}^m} + (N(w), X w)_{\mathbf{R}^m} \big)^{\frac{p-2}{p - 1}} | X v - X w |^{\frac{1}{p-1}}
\end{equation}
for every $v, w \in V$. (To get an idea of how to prove \eqref{pastaconlabottarga} and \eqref{pasta2}, we remind the interested reader to the equivalent estimates in the parabolic case, namely \eqref{classM2} and \eqref{due}, provided in the proof of Theorem \ref{MainTh}.)
\medskip

\noindent\textbf{Step 2.} Let us show that $A$ in an operator in $X$-divergence form, by providing the existence of a function $a (x,\cdot) : \R^m \to \R^m$ satisfying
\begin{align}\label{lacito_e}
N (u) = a (x, Xu)\quad\text{for any }u\in V\text{ a.e. }x\in\Omega\,.
\end{align}

Fix an open set $\omega$ such that $\overline{\omega} \subset \Omega$, let $\phi \in {\bf C}^1_0 (\Omega)$ be such that $\phi \equiv 1$ in $\omega$ and, for any $\xi\in\R^n$, define
\begin{equation*}
    w_{\xi} (x) := \big( \xi, x \big)_{\mathbf{R}^n} \phi (x)\quad\text{for any }x\in\omega\,,
\end{equation*}
where $( \xi, x)_{\mathbf{R}^n} = \xi_1 x_1 + \ldots + \xi_n x_n$. Then
\begin{equation}\label{gradXaffine2}
    X w_{\xi} (x) = C(x) \xi \quad \text{for any } x \in \omega\,,
\end{equation}
where $C(x)$ is the coefficient matrix of the $X$-gradient (see Section \ref{sect2}).

\noindent Notice that, if $\xi, \tilde{\xi} \in \R^n$ are such that $\xi - \tilde{\xi} \in $ Ker$\, C(x)$, then
\begin{align}\label{serve}
N w_{\xi} (x) = N w_{\tilde\xi} (x) \quad \text{for any } x \in \omega\,.
\end{align}
Indeed, by \eqref{pasta2}, it holds that
\begin{align*}
    | N w_{\xi} (x) - N w_{\tilde\xi} (x) |  &\leq \beta' ( 1 + ( N w_{\xi} , X w_{\xi} )_{\R^m} 
	+ (N w_{\tilde\xi}, X w_{\tilde\xi} )_{\R^m})^{\frac{p-2}{p-1}}\\
	&\quad\times|X w_{\xi} - X w_{\tilde\xi} |^{\frac{1}{p-1}}
\end{align*}
and, by \eqref{gradXaffine2}, the right hand side, and then the left hand side, is zero. 

\noindent Fix $\xi \in \R^n$. By condition (LIC), there exists a Lebesgue measure zero subset of $\Omega$, $Z_X$, such that $(X_1(x),\dots,X_m(x))$ are linearly independent for any $x \in \omega \setminus Z_X$.
Denote by $\eta = \eta (x)$ the vector $C(x) \xi \in \R^m$ and define
\begin{equation*}
    \tilde{a} (x, \eta (x) ) := N w_{\xi} (x)\,.
\end{equation*}
By \eqref{serve}, $\tilde{a}$ is well-defined and, by condition (LIC), we can define $\tilde{a} (x, \cdot)$ in the whole space $\R^m$. 

\noindent Consider now a sequence of open sets $(\omega_j)_j$ such that
$\overline{\omega}_j \subset \Omega$ for every $j \in {\bf N}$ and such that $\cup_{j=1}^\infty \omega_j = \Omega$. Moreover, let $(\phi_j)_j\subset {\bf C}^1_0(\Omega)$ be such that
\[
\phi_j \equiv 1 \quad \text{on } \omega_j \quad \text{for every } j \in {\bf N},
\]
and, for any fixed $\xi\in\mathbf{R}^n$, define the maps
\begin{align*}
    w_{\xi}^j (x) &:= \big( \xi, x \big)_{\mathbf{R}^n} \phi_j  (x) \quad \text{for any } x\in\omega_j\,,\\
    a (x, \eta (x)) &:= \lim_{j \to +\infty} N w_{\xi}^j (x)\quad \text{for any } x\in \Omega \setminus Z_X\,,
\end{align*}
where $C(x) \xi = \eta$. Then, by previous considerations on $\tilde{a}$, \eqref{lacito_e} follows.
\medskip

\noindent\textbf{Step 3.} We conclude by showing that $a \in \tilde{\mathcal{M}}_{\Omega}(\alpha, \beta', p)$. 

By construction, $a (x,\cdot) : \R^m \to \R^m$ turns out to be continuous for a.e. $x \in \Omega$ and $a (\cdot, \eta) : \Omega \to {\bf R}^m$ turns out to be measurable for every $\eta \in {\bf R}^m$. Then, $a : \Omega\times\R^m \to \R^m$ is a Carath\'eodory function.
Moreover, by \eqref{pastaconlabottarga}, $a$ satisfies condition $(ii)$ of Definition \ref{M,alpha,beta} and, by \eqref{pasta2}, condition $(iii)'$ of Remark \ref{chebelfreschetto}.

\noindent Let us prove that $a (\cdot, 0) = 0$. 
Let $w_h$ be the solution to \eqref{E_h}, with $g = 0$. Since, by Definition \ref{M,alpha,beta} (i), $a_h (x, 0) = 0$ a.e. in $\Omega$ then, by the uniqueness of the solution to problem \eqref{E_h}, we get that
$w_h \equiv 0$ a.e. in $\Omega$ for every $h \in {\bf N}$ and, by Lemma \ref{lem6.5_ell} and the definition of $M$, we get (up to subsequences)
\[
M(0) = \lim_{h \to +\infty} a_h (x, X w_h) = \lim_{h \to +\infty} 0 = 0\,,
\]
that is, $N(0) = 0$.

\noindent Fix $\bar\xi = 0\in\mathbf{R}^n$. Then, $N w_{\bar\xi}^j (x) = 0$ in $\omega_j$ for any $j\in\mathbf{N}$ and, by the definition of $a$, we finally get $a(x, 0) = 0$ a.e. $x\in\Omega$.
\end{proof}


\section{Parabolic \g-convergence}\label{sec.MainResult}


The last section of this paper is devoted to the study of the $\g$-convergence of sequences of parabolic operators depending on vector fields. As done in the elliptic case (Section \ref{elliptic}), we provide in Theorem \ref{MainTh} a $G$-compactness theorem, Theorem \ref{MainTh}, which follows from preliminary results, Theorem \ref{th5.10}, Theorem \ref{div} and Lemma \ref{lem6.5}.
We conclude this section by showing in Lemma \ref{lemmafinale} that, whenever the sequence of Carath\'eodory functions $(a_h)_h$, that defines the monotone parabolic operators $(\partial_t+\mathrm{div}_X(a_h(x,X)))_h$, does not depend on $t$ for every $h\in\mathbf{N}$, then the parabolic $\g$-limit is the operator
\[
\partial_t+\mathrm{div}_X(a(x,X))\,,
\]
where $\mathrm{div}_X(a(x,X))$ is the elliptic $\g$-limit of the sequence of operators $(\mathrm{div}_X(a_h(x,X)))_h$.
In all this section we always assume that $\Omega$ is a bounded domain of $\mathbf{R}^n$, that $2\leq p<\infty$ and that $X$ satisfies conditions (H1), (H2), (H3) and (LIC), given in the Introduction.
\medskip

The first result of this section, which is proved in \cite[Lemma 3]{zko}, shows that the sequence of solutions $(u_h)_h$ to problems \eqref{P_h}, that are naturally compact in $L^p (0, T; H)$ as stated in Proposition \ref{propW}, converges to its limit in the space $\mathbf{C}^0([0,T];H)$.
\begin{theorem}[{\cite[Lemma 3]{zko}}]\label{th5.10}
Let $u_{h} \in \mathcal{W}$ be the solution to problem \eqref{P_h} and let $u$ be the limit, up to subsequences, of $(u_{h})_h$ in $L^p(0,T;H)$. Then,
\[
u_{h} \to u \quad\text{in } \mathbf{C}^0([0,T];H)\quad\text{as }h\to\infty\,.
\]
\end{theorem}
\begin{theorem}\label{div}
Let $v_h,v,w_h,w \in \mathcal{W}$ satisfy
\begin{align*}
v_h \to v \quad\text{ weakly in }	\mathcal{V} \,, \qquad 
v_h' \to v' \quad\text{ weakly in }	\mathcal{V}' \,,				\\
w_h \to w \quad\text{ weakly in }	\mathcal{V} \,, \qquad 
w_h' \to w' \quad\text{ weakly in }	\mathcal{V}'
\end{align*}
and assume that $(M_h)_h , (N_h)_h \subset L^{p'}(0,T;L^{p'}(\Omega; \R^m))$ weakly converge to $M$ and to $N$ in $L^{p'}(0,T;L^{p'} (\Omega; \R^m))$, respectively.
Suppose that
\begin{align*}
v_h' + \text{\rm div}_X \, M_h &= f \qquad \text{in } \mathcal{V}' \,,			\\	
w_h' + \text{\rm div}_X \, N_h &= g \qquad \text{in } \mathcal{V}'
\end{align*}
for some $f , g \in \mathcal{V}'$. Then,
\[
(M_h - N_h ,X v_h - Xw_h)_{\R^m} \to (M - N, Xv - Xw)_{\R^m}\quad\text{in }\mathcal{D}'(\Omega\times(0,T)) \, .
\]
\end{theorem}
\begin{proof}
The proof can be obtained in a way similar to the analogous one showed in the elliptic case, namely Theorem \ref{comp_comp}.
For reader's convenience we provide in the following the main calculations.
\medskip

Fix $\varphi \in C^{\infty}_c(\Omega \times (0,T))$ and consider the quantity
\[
\langle(v_h' + \text{\rm div}_X M_h)-(w_h' + \text{\rm div}_X N_h),(v_h-w_h)\,\varphi\rangle_{\mathcal{V}'\times\mathcal{V}}\,.
\]
Then
\begin{align*}
    \int_0^T\int_{\Omega} & (M_h-N_h, X w_h - Xw_h)_{\R^m} \,\varphi \,dx dt \\
    &=\langle (v_h' + \text{\rm div}_X M_h)-(w_h' + \text{\rm div}_X N_h), (v_h-w_h)\,\varphi \rangle_{\mathcal{V}'\times\mathcal{V}}\\
    &\quad - \langle (v_h' - w_h'), (v_h-w_h)\,\varphi\rangle_{\mathcal{V}'\times\mathcal{V}} - \int_0^T \int_{\Omega}  (M_h-N_h, X \varphi)_{\R^m}(v_h - w_h)\, dx dt\,. 
\end{align*}
Consider the right hand side terms.
By assumptions,
\begin{align*}
    \langle (v_h' + \text{\rm div}_X M_h) & - (w_h' + \text{\rm div}_X N_h), (v_h-w_h)\,\varphi \rangle_{\mathcal{V}'\times\mathcal{V}}\\
    & = \langle f - g, (v_h-w_h) \,\varphi \rangle_{\mathcal{V}'\times\mathcal{V}} \to \langle f -g , (v-w) \,\varphi \rangle_{\mathcal{V}'\times\mathcal{V}}\,.
\end{align*}
As regards the second term, by Proposition \ref{propW} and Theorem \ref{th5.10}, we get
\begin{align*}
    2\, \langle (v_h' - w_h')&, (v_h-w_h)\, \varphi\rangle_{\mathcal{V}'\times\mathcal{V}}\\
	& = \int_{\Omega} (v_h-w_h)^2 (x,T) \, dx - \int_{\Omega} (v_h-w_h)^2 (x,0) \, dx\\
	& \quad - \int_0^T \!\!\! \int_{\Omega} (v_h-w_h)^2 \varphi' (x,t) \, dx dt\\
	& \to \int_{\Omega} (v - w)^2 (x,T) \, dx - \int_{\Omega} (v - w)^2 (x,0) \, dx\\
	& \quad - \int_0^T \!\!\! \int_{\Omega} (v - w)^2 \varphi' (x,t) \, dx dt\\
	& = 2\, \langle (v' - w'), (v - w) \,\varphi \rangle_{\mathcal{V}'\times\mathcal{V}}\,.
\end{align*}
For the third term one can proceed as in the proof of Theorem \ref{comp_comp} and conclude.
\end{proof}

The proof of the next result is classical and can be obtained following and adapting (since this is for $p = 2$) the analogous one contained in \cite{spagnolo} or, for $p \geq 2$, the proof contained in \cite{svan2}.
\begin{lem}\label{lem6.5}
Denote by $u_{h} (f, \varphi)$ the solution to problem \eqref{P_h}, $h\in\mathbf{N}$.
Then, there exist three continuous operators
\begin{align*}
    \mathcal{B}: \mathcal{V}' \times H &\to \mathcal{W} \,,\\
    \mathcal{K}: \mathcal{V}' \times H &\to \mathcal{V}' \,,\\
    \mathcal{M}: \mathcal{V}'  \times H &\to L^{p'} (0,T ; L^{p'}(\Omega;\R^m))
\end{align*}
such that, up to subsequences,
\begin{align*}
    u_{h} (f,\varphi) \to \mathcal{B} (f,\varphi) \quad &\text{in }L^p(0,T; H)\,,\\
    \mathcal{A}_h u_{h} (f,\varphi) \to \mathcal{K} (f,\varphi) \quad &\text{in }  \mathcal{V}'\,,\\
    a_h (\cdot,\cdot, X u_{h} (f,\varphi))\to \mathcal{M} (f,\varphi)	\quad &\text{weakly in } L^{p'} (0,T ; L^{p'}(\Omega;\R^m))\,.
\end{align*}
Moreover, $\mathcal{B}$ is injective, $\mathcal{B} (\mathcal{V}' \times H)$ is dense in $\mathcal{W}$ and  $\mathcal{B}$, $\mathcal{K}$ and $\mathcal{M}$ satisfy
\begin{align*}
\partial_t \big( \mathcal{B} (f,\varphi) \big) + \mathcal{K} (f,\varphi) = f	\quad  &\textit{in } \mathcal{V}' \,,\\
\mathcal{K} (f,\varphi) = \text{\rm div}_X \mathcal{M} (f,\varphi) \quad  &\textit{in } \mathcal{V}' \,,\\
\mathcal{B} (f,\varphi)(0) = \varphi \quad&\textit{in }H
\end{align*}
for every $f \in \mathcal{V}'$ and $\varphi \in H$.
\end{lem}
\begin{theorem}\label{MainTh}
Consider a sequence $(a_h)_h\subset\mathcal{M}_{\Omega\times(0,T)}(\alpha,\beta,p)$ and the related sequence of parabolic operators $(\mathcal{P}_{h})_{h}$, defined in \eqref{operatori}.

\noindent Then, there exists $a \in \tilde{\mathcal{M}}_{\Omega\times(0,T)}(\alpha, \beta', p)$, with $\beta' = \left( \alpha^{-1} \beta^p \right)^{\frac{1}{p-1}}  \max \{ 1 , \alpha^{-\frac{p-2}{p-1}} \}$,
such that, up to subsequences,
\[
\mathcal{P}_{h}\quad \g\text{-converges to } \mathcal{P}\,,
\]
where $\mathcal{P}:\mathcal{W}\to\mathcal{V}'$ is the operator in $X$-divergence form associated to $a$ and defined by
\[
\mathcal{P}(u)=u'+\mathrm{div}_X (a (\cdot,\cdot, Xu ))\quad\text{for every }u\in\mathcal{W}\,.
\]
\end{theorem}
\begin{proof}
By Lemma \ref{lem6.5}, $\mathcal{B}$ is injective. Therefore, we can define the operator
\begin{align*}
\mathcal{A} : \mathcal{B} (\mathcal{V}' \times H) \to \mathcal{V}' , \qquad 
	\mathcal{A} \big( \mathcal{B} (f, \varphi) \big) := \mathcal{K} (f, \varphi)\,,
\end{align*}
which satisfies
\[
\alpha\| v - w \|_{\mathcal{V}}^p	\leq \big\langle \mathcal{A} v -  \mathcal{A} w \, ,  v - w \big\rangle_{\mathcal{V}' \times \mathcal{V}}
\]
and, by similar arguments of Remark \ref{chebelfreschetto}, also
\begin{align*}
\| \mathcal{A} v - \mathcal{A} w \|_{\mathcal{V}'} \leq \tilde\beta
	\Big( |\Omega| + \| v \|_{\mathcal{V}}^{p} + \| w \|_{\mathcal{V}}^{p} \Big)^{\frac{p-2}{p-1}} \| v - w \|_{\mathcal{V}}^{\frac{1}{p-1}}
\end{align*}
for every $v, w \in \mathcal{B} (\mathcal{V}' \times H)$, where $\tilde\beta = \big( \beta \, \alpha^{-\frac{1}{p}} \big)^\frac{1}{p-1}$.

\noindent Since $\mathcal{B} (\mathcal{V}' \times H)$ is dense in $\mathcal{W}$, and then also in $\mathcal{V}$, by the previous estimates $\mathcal{A}$ can be extended to another operator, still denoted by $\mathcal{A}$,  $\mathcal{A} : \mathcal{V} \to \mathcal{V}'$, satisfying the above estimates and, by Lemma \ref{lem6.5}, also
\begin{align*}
\big(\mathcal{B} (f,\varphi) \big)' + \mathcal{A} \big( \mathcal{B} (f,\varphi) \big) = f \quad  &\textit{ in } \mathcal{V}' \, ,		\\
\mathcal{A} \big( \mathcal{B} (f,\varphi) \big) = \text{\rm div}_X \mathcal{M} (f,\varphi) \quad  &\textit{ in } \mathcal{V}' \, .
\end{align*}
Denote
\[
\mathcal{P} u := u' + \mathcal{A} u , \qquad u \in \mathcal{W}\,,
\]
\[
\begin{array}{cccc}
\tilde{\mathcal{P}} :	&	\mathcal{W}	&	\to		&	\mathcal{V}' \times H				\\
	\			&	u			&	\mapsto	&	\mathcal{B}^{-1} u
\end{array}
\]
and, more specifically,
\[
\mathcal{B}^{-1} u = \big( \mathcal{P} u , u (0) \big)\,,
\]
and consider the composition
$\mathcal{N} := \mathcal{M} \circ {\tilde{\mathcal{P}}} : \mathcal{W}\to L^{p'} (0,T ; L^{p'}(\Omega;\R^m))$.
Let us show that
\begin{equation}\label{classM2}
    (\mathcal{N} u - \mathcal{N} v, Xu - Xv)_{\R^m} \geq \alpha |Xu - Xv|^p \, ,
\end{equation}
and
\begin{align}\label{due}
    |\mathcal{N} u-\mathcal{N} v| \leq 
    \beta' \left[1+ (\mathcal{N} u,Xu)_{\R^m}+(\mathcal{N} v,Xv)_{\R^m}\right]^\frac{p-2}{p-1}|Xu-Xv|^\frac{1}{p-1}
\end{align}
a.e. in $\Omega\times(0,T)$ for any $u,v \in \mathcal{W}$, where $\beta' = {\displaystyle \left( \alpha^{-1}\,\beta^p \right)^{\frac{1}{p-1}}  \max \{ 1 , \alpha^{-\frac{p-2}{p-1}} \}}$.

\noindent Fix $u,v\in\mathcal{W}$, let $(f,\varphi),(g,\psi)\in\mathcal{V}'\times H$ satisfy
\[
u(x,t) = \mathcal{B} (f,\varphi)\quad\text{and}\quad v(x,t) = \mathcal{B} (g,\psi)
\]
for any $(x,t)\in\Omega\times(0,T)$ and denote
\begin{align*}
u_{h} &:= \mathcal{P}_{h}^{-1}(\tilde{\mathcal{P}} u)\,,\\
v_{h} &:= \mathcal{P}_{h}^{-1}(\tilde{\mathcal{P}} v)\,.
\end{align*}
By Lemma \ref{lem6.5},
\begin{equation*}
\begin{split}
    a_h(\cdot,\cdot,Xu_{h} (f,\varphi))&\to \mathcal{M} (f,\varphi) = \mathcal{M} (\tilde{\mathcal{P}} u)=\mathcal{N} u \, ,		\\
    a_h(\cdot,\cdot,Xv_{h} (g,\psi))&\to\mathcal{M} (g,\psi) = \mathcal{M} (\tilde{\mathcal{P}} v)=\mathcal{N} v
\end{split}
\end{equation*}
weakly in $L^{p'} (0,T ; L^{p'}(\Omega;\R^m))$ and,
since, by Definition \ref{M,alpha,beta} $(ii)$,
\begin{equation*}
\begin{split}
(a_h(x,t,Xu_{h})-a_h(x,t,Xv_{h})&,Xu_{h}-Xv_{h})_{\R^m} \geq \alpha |Xu_{h} - Xv_{h}|^p\quad h\in\mathbf{N}\,,
\end{split}
\end{equation*}
then, by passing to the limit, we get \eqref{classM2} in virtue of Lemma \ref{chiado'} and Theorem \ref{div}.
Moreover, by Definition \ref{M,alpha,beta} (ii) and (iii), it also holds that
\begin{align*}
    &|a_h(x,t,Xu_{h}) - a_h(x,t,Xv_{h})| \\
    &\leq \beta \left[ 1 + \alpha^{-1} (a_h(x,t,Xu_{h}),Xu_{h})_{\R^m} + \alpha^{-1} (a_h(x,t,Xv_{h}),Xv_{h})_{\R^m}\right]^{\frac{p-2}{p}} \\
    &\quad \times \alpha^{-\frac{1}{p}} (a_h(x,t,Xu_{h})-a_h(x,t,Xv_{h}), Xu_{h}-Xv_{h})_{\R^m}^{\frac{1}{p}}\quad h\in\mathbf{N}\,.
\end{align*}
Then, by Lemma \ref{chiado'}, Theorem \ref{div} and Lemma \ref{lem6.5} and, by passing to the limit, we get
\begin{align*}
    &|\mathcal{M} (f,\varphi) - \mathcal{M} (g,\psi)|^{\frac{p-1}{p}}\\
    &\leq\beta\alpha^{-\frac{1}{p}} ( 1 + \alpha^{-1} \big((\mathcal{M} (f,\varphi),X\mathcal{B} (f,\varphi))_{\R^m} + (\mathcal{M} (g,\psi),X\mathcal{B} (g,\psi))_{\R^m}\big))^{\frac{p-2}{p}}\\
    & \quad \times |X\mathcal{B} (f,\varphi)-X\mathcal{B} (g,\psi)|^{\frac{1}{p}}
\end{align*}
and \eqref{due} follows too.


We conclude showing that $\mathcal{A}$ is an operator in $X$-divergence form, that is, by constructing a limit function $a(x,t,\cdot) : \R^m \to \R^m$ satisfying
\begin{align*}
    \mathcal{N} u = a (x,t, Xu)\quad\text{for any }u\in \mathcal{W}\text{ a.e. }(x,t)\in\Omega\times(0,T)
\end{align*}
and, finally, by showing that $a \in \tilde{\mathcal{M}}_{\Omega\times(0,T)}(\alpha, \beta', p)$.


\noindent Fix sequences of open sets $\omega_j$ and $I_j$ such that $\overline{\omega_j} \subset \Omega$ and $\overline{I_j}\subset(0,T)$ for every $j\in\mathbf{N}$ and satisfying $\cup_{j\in\mathbf{N}}\omega_j\times I_j=\Omega\times(0,T)$, and a sequence of cut-off functions $\Phi_j \in {\bf C}^1_0 (\Omega\times(0,T))$ such that $\Phi_j(x,t) \equiv 1$ in $\omega_j\times I_j$ for any $j\in\mathbf{N}$. Then, for any fixed $\xi\in\mathbf{R}^n$, denoted by $\eta=\eta(x)$ the vector such that $C(x) \xi = \eta\in\mathbf{R}^m$, we define
\begin{align*}
    a (x,t,\eta (x)):=\lim_{j \to +\infty} \mathcal{N} w_{\xi}^j (x)\quad \text{for any } x\in \Omega \setminus Z_X\,,
\end{align*}
where, set $(\xi, x)_{\mathbf{R}^n} = \xi_1 x_1 + \ldots + \xi_n x_n$, the map $w_{\xi}^j$ is defined by
\begin{align*}
    w_{\xi}^j (x,t):= \big( \xi, x \big)_{\mathbf{R}^n} \Phi_j  (x,t) \quad \text{for any } (x,t)\in\omega_j\times I_j\,.
\end{align*}
Notice that $a$ is well-defined by condition (LIC) on the $X$-gradient, which ensures that $(X_1(x),\dots,X_m(x))$ are linearly independent outside $Z_X$.

\noindent The last part of the proof follows verbatim as in Step 2 and Step 3 of Theorem \ref{MainTh_1}, where one uses \eqref{classM2} and \eqref{due}, instead of \eqref{pastaconlabottarga} and \eqref{pasta2}, and Lemma \ref{lem6.5}, instead of Lemma \ref{lem6.5_ell}. 
\end{proof}
\begin{lem}\label{lemmafinale}
Consider a sequence $(a_h)_h \subset \mathcal{M}_{\Omega} (\alpha,\beta,p)$. If
\begin{align*}
\text{\rm div}_X (a_h (x, X)) \quad &\text{$\g$-converges to} \quad \text{\rm div}_X (a (x, X)) \text{ and }\\
\partial_t + \text{\rm div}_X (a_h (x, X)) \quad &\text{$\g$-converges to} \quad \partial_t + \text{\rm div}_X (b (x, t, X))
\end{align*}
then,
\[
a = b .
\]
\end{lem}
\begin{proof}
Fix $g \in V'$ and denote by $u_h (g)\in V$ the solution to problem
\[
\text{\rm div}_X (a_h (x, Xw)) = g\quad\text{in }V'\,.
\]
Then, by hypotheses,
\begin{align*}
    u_h (g)\to u(g)\quad &\text{strongly in }L^p(\Omega)\,,\\
    a_h(\cdot,Xu_h(g))\to a(\cdot, Xu(g))\quad&\text{weakly in }L^{p'}(\Omega ; \R^m)\,.
\end{align*}
In a similar way, for any fixed $f \in \mathcal{V}'$ and $\varphi \in H$, denote by $v_h (f,\varphi)\in\mathcal{W}$ the solution to problem
\begin{equation}\label{burano}
\begin{cases}
    w' + \text{\rm div}_X (a_h (x, Xw)) = f&\text{ in }\mathcal{V}'\\
    w(0) = \varphi&\text{ in }H			
\end{cases}.
\end{equation}
Let $f(x,t) := g(x)$ for any $t\in [0,T]$ and define $\varphi_h := u_h (g)$ for any $h\in\mathbf{N}$.
Since, by Remark \ref{ultimanota} $(\varphi_h)_h$ strongly converges to $\varphi:=u (g)$ in $H$, then
\begin{align}\label{pippo2}
    v_h (g, u_h (g))\to v (g, u(g))\quad &\text{strongly in }	L^p (0,T; L^2 (\Omega))\,,\\
    \label{pippo}
    a_h (\cdot, X v_h (g, u_h (g)))\to b(\cdot, \cdot, X v (g, u (g)))\quad&\text{weakly in }L^{p'} (0,T; L^{p'}(\Omega ; \R^m))\,.
\end{align}
Moreover, since $u_h(g)$ is also a solution to problem \eqref{burano}, with fixed data $f(x,t) := g(x)$ for any $t\in [0,T]$ and $\varphi_h := u_h (g)$ and, since $v_h (0) = u_h (g)$, then, by the uniqueness of the solution to problem \eqref{burano}, it holds that
\[
v_h (x, t) = u_h (x) \quad \text{for any } t \in [0,T]
\]
and, by \eqref{pippo2}, that  $v (g, u(0) ) = u (g)$, that is, $v$ turns out to be independent of $t$.
By \eqref{pippo}, since the left hand side term is independent of $t$, we have that
$$
b (x, t, \xi) = b (x, \xi) \, .
$$
Therefore,
\begin{align*}
    \text{\rm div}_X (a (x, X u (g) )) = g \quad&\text{in } V'\,,\\
    \text{\rm div}_X (b (x, t,X v (g, u(0) )  )) = \text{\rm div}_X (b (x, X u (g))) = g \quad&\text{in } V'\,,
\end{align*}
that is,
\[
A^{-1} g = B^{-1} g \qquad \text{ for every } g \in V',
\]
where $A:V\to V'$ and $B:V\to V'$ are, respectively, defined by
\[
A w := \text{\rm div}_X (a (x, X w )) \quad \text{and} \quad B w := \text{\rm div}_X (b (x, X w ))\,,\quad w\in V\,.
\]
Then, $A = B$ and, by the convergence of momenta, we finally get $a = b$.
\end{proof}



\begin{thebibliography}{100}


\bibitem{BFTT} 
{\sc A. Baldi, B. Franchi, N. Tchou, M. C. Tesi},
{\em Compensated compactness for differential forms in {C}arnot groups and applications},
{Adv. Math.} \textbf{223} (2010), no. 5, 1555--1607.

\bibitem{cp-dm-df} {\sc V. {Chiad\`o Piat}, G. {Dal Maso}, A. Defranceschi}, {\em {G}-convergence of monotone operators}, {Ann. Inst. H. Poincar\'e, Anal. Non Lin\'eaire} \textbf{7} (1990), no. 3, 123--160.

\bibitem{col-spa} {\sc F. Colombini, S. Spagnolo}, {\em Sur la convergence de solutions d'\'equations paraboliques}, J. Math. Pures Appl. (9) \textbf{56} (1977), no. 3, 263--305.

\bibitem{dalmaso} {\sc G. {Dal Maso}}, {\em An introduction to $\Gamma$-convergence},
Birkh\"auser, Boston (1993).

\bibitem{degiorgi-fran} {\sc E. {De Giorgi}, T. Franzoni}, {\em Su un tipo di convergenza variazionale}, Atti Accad. Naz. Lincei Rend. Cl. Sci. Fis. Mat. Nat. (8) \textbf{58} (1975), no. 6, 842--850.

\bibitem{deg-sp} {\sc E. {De Giorgi}, S. Spagnolo}, {\em Sulla convergenza degli integrali dell'energia per operatori ellittici del secondo ordine}, Boll. Un. Mat. Ital. (4) \textbf{8} (1973), 391--411.

\bibitem{EPV}
{\sc F. Essebei, A. Pinamonti, S. Verzellesi},
{\em Integral representation of local functionals depending on vector fields}, to appear in Adv. Calc. Var. (2022) \url{https://doi.org/10.1515/acv-2021-0054}.

\bibitem{EV}
{\sc F. Essebei, S. Verzellesi},
{\em $\Gamma$-Compactness of Some Classes of Integral Functionals Depending on Vector Fields}, submitted. Preprint available at \url{https://arxiv.org/pdf/2112.05491.pdf}.

\bibitem{FS} {\sc G.B. Folland, E.M. Stein},
{\em Hardy spaces on homogeneous groups}, 
Mathematical Notes {\bf 28}, Princeton University Press, Princeton, N.J.; University of Tokyo Press, Tokyo, 1982.

\bibitem{FGVN}
{\sc B. Franchi, C.E. Guti\'{e}rrez, T. van Nguyen},
{\em Homogenization and convergence of correctors in Carnot groups}, Comm. Partial Differential Equations \textbf{30} (2005), no. 10-12, 1817--1841.

\bibitem{FTT} 
{\sc B. Franchi, N. Tchou, M.C. Tesi},
{\em Div-curl type theorem, {$H$}-convergence and {S}tokes formula in the {H}eisenberg group},
{Commun. Contemp. Math.} \textbf{8} (2006), no. 1, 67--99.

\bibitem{FT}
{\sc B. Franchi, M.C. Tesi}, 
{\em Two-scale homogenization in the Heisenberg group},  J. Math. Pures Appl. (9) \textbf{81} (2002), 495--532.

\bibitem{FSSC1} {\sc B. Franchi, R. Serapioni, F. Serra Cassano}, {\em Meyers-{S}errin type theorems and relaxation of variational integrals depending on vector fields}, Houston J. Math. {\bf 22} (1996), no. 4, 859--890.

\bibitem{FSSC2} {\sc B. Franchi, R. Serapioni, F. Serra Cassano} {\em Approximation and imbedding theorems for weighted {S}obolev spaces associated with {L}ipschitz continuous vector fields}, Boll. Un. Mat. Ital. B (7) {\bf 11} (1997), no. 1, 83--117.

\bibitem{GN} {\sc N. Garofalo, D.-M. Nhieu}, {\em Isoperimetric and {S}obolev inequalities for {C}arnot-{C}arath\'{e}odory spaces and the existence of minimal surfaces}, Comm. Pure Appl. Math. {\bf 49} (1996), no. 10, 1081--1144.

\bibitem{M} {\sc A. Maione}, {\em H-convergence for equations depending on monotone operators in Carnot groups}, Electron. J. Differential Equations {\bf 2021} (2021), no. 13, 1--13.

\bibitem{MPSC1} {\sc A. Maione, A. Pinamonti, F. Serra Cassano},
{\em $\Gamma$-convergence for functionals depending on vector fields. I. Integral representation
and compactness}, J. Math. Pures Appl. (9) \textbf{139} (2020), 109--142.

\bibitem{MPSC2} {\sc A. Maione, A. Pinamonti, F. Serra Cassano}, {\em $\Gamma$-convergence for functionals depending on vector fields. II. Convergence of minimizers}, submitted. Preprint available at \url{https://cvgmt.sns.it/paper/5107/}.

\bibitem{MV} {\sc A. Maione, E. Vecchi}, {\em Integral representation of local left--invariant functionals in Carnot groups}, Anal. Geom. Metr. Spaces \textbf{8} (2020), no. 1, 1--14.

\bibitem{MS} {\sc N.G. Meyers, J. Serrin}, {\em H = W}, Proc. Nat. Acad. Sci. USA {\bf 51} (1964), 1055--1056.

\bibitem{pankov} {\sc A. Pankov}, {\em {$G$}-convergence and homogenization of nonlinear partial differential operators}, Mathematics and its Applications \textbf{422}, Kluwer Academic Publishers, Dordrecht (1997).

\bibitem{show2} {\sc R.E. Showalter}, {\em Monotone operators in Banach space and nonlinear partial differential equations}, (English summary) Mathematical Surveys and Monographs, 49. American Mathematical Society, Providence, RI, 1997.

\bibitem{simon} {\sc L. Simon}, {\em On {$G$}-convergence of elliptic operators}, Indiana Univ. Math. J. \textbf{28} (1979), no. 4 587--594.

\bibitem{spagnolo} {\sc S. Spagnolo}, {\em Convergence of Parabolic Equations},
Boll. Un. Mat. Ital. B (5) \textbf{14} (1977), no. 2 547--568.

\bibitem{spagnolo1} {\sc S. Spagnolo}, {\em Una caratterizzazione degli operatori differenziali autoaggiunti del $2^{\circ}$ ordine a coefficienti misurabili e limitati},
Rend. del Sem. Mat. dell'Univ. di Padova \textbf{39} (1967), 56--64.

\bibitem{spagnolo2} {\sc S. Spagnolo}, {\em Sul limite delle soluzioni di problemi di Cauchy relativi all'equazione del calore},
Ann. Scuola Norm. Sup. Pisa Cl. Sci. (3) \textbf{21} (1967), 657--699.

\bibitem{spagnolo3} {\sc S. Spagnolo}, {\em Sulla convergenza di soluzioni di equazioni paraboliche ed ellittiche}, Ann. Scuola Norm. Sup. Pisa (3) 22 (1968), 571--597; errata, ibid. (3) \textbf{22}, 673.

\bibitem{svan2} {\sc N. Svanstedt}, {\em {G}-convergence of parabolic operators}, Nonlinear Anal. \textbf{36} (1999), no. 7 Ser. A: Theory Methods, 807--842.

\bibitem{tartar-murat} {\sc L. Tartar}, {\em Cours Peccot}, {Coll\`ege de France}, 1977.
\newblock Partially written in: {\sc F. Murat,} {\it H-convergence - S\'eminaire d'Analyse Fonctionnelle et Num\'erique}, Universit\'e\ d'Alger, 1977-78. English translation: {\sc F.\ Murat - L.\ Tartar:} {\it H-Convergence}, in {\it Topics in the Mathematical Modelling of Composite Materials} (A. Cherkaev, R. Kohn, editors), Birkh\"auser, Boston, 1997, 21--43.

\bibitem{T} {\sc L. Tartar}, {\em The general theory of homogenization, A personalized introduction}, Lecture Notes of the Unione Matematica Italiana \textbf{7},
\emph{Springer-Verlag, Berlin; UMI, Bologna}, 2009.

\bibitem{Z2A} {\sc E. Zeidler}, {\em Nonlinear functional analysis and its applications. {II}/{A}, {II}/{B}}, Springer-Verlag, New York, 1990.


\bibitem{zko} {\sc V.V. \v{Z}ikov, S.M. Kozlov, O.A. Ole\u{\i}nik}, {\em {$G$}-convergence of parabolic operators}, Uspekhi Mat. Nauk \textbf{36} (1981), no. 1(217), 11--58, 248.

\bibitem{iquattrorussi} {\sc V.V. \v{Z}ikov, S.M. Kozlov, O.A. Ole\u{\i}nik, Kha T{\textquotesingle}en Ngoan}, {\em {Averaging and {$G$}-convergence of differential operators}}, Uspekhi Mat. Nauk \textbf{34} (1979), no. 5(209), 65--133, 256.

\end{thebibliography}
\end{document}